\documentclass[a4paper,11pt,twoside]{article}
\usepackage{latexsym,amsfonts,index,amssymb}
\usepackage[centertags]{amsmath}
\usepackage{amstext}
\usepackage{enumerate}
\usepackage{cases}
\usepackage{color}
\usepackage{enumitem}
\usepackage{bbold}
\usepackage{graphicx}
\usepackage{stmaryrd}
\usepackage{calrsfs}

\usepackage[arrow,curve,rotate]{xypic}

\DeclareSymbolFont{rsfscript}{OMS}{rsfs}{m}{n}
\DeclareSymbolFontAlphabet{\mathrsfs}{rsfscript}
\renewcommand{\mathcal}{\mathrsfs}

\def\titlerunning#1{\gdef\titrun{#1}}
\makeatletter
\def\author#1{\gdef\autrun{\def\and{\unskip, }#1}\gdef\@author{#1}}
\def\address#1{{\def\and{\\\hspace*{18pt}}\renewcommand{\thefootnote}{}%
\footnote {#1}}%
\markboth{\autrun}{\titrun}} \makeatother
\def\email#1{e-mail: #1}

\def\keywords#1{\par\medskip
\noindent\textbf{Keywords.} #1}

\setenumerate[1]{label=$(\alph*)$,leftmargin=*}
\setenumerate[2]{label=$(\roman*)$,leftmargin=*}

\DeclareMathAlphabet{\mathpzc}{OT1}{pzc}{m}{it}
\definecolor{verde}{rgb}{0,.5,0}

\DeclareSymbolFont{bbold}{U}{bbold}{m}{n}
\DeclareSymbolFontAlphabet{\mathbbold}{bbold}

\numberwithin{equation}{section}

\def\N{{\mathbb N}}

\newfont{\sss}{cmssi10 at 11pt}
\newfont{\bss}{cmssbx10 at 11pt}
\newfont{\tit}{cmitt10 at 11pt}
\newcommand{\K}{{\bf K}}
\newcommand{\D}{{\bf D}}

\newcommand{\V}{{\bf V}}

\newcommand{\AN}{A^{\N}}
\newcommand{\NA}{A^{-\N}}

\newcommand{\mfrg}[3] {{#1}:{#2}\mathop{\hbox{\kern5pt$\circ$\kern-12pt\raise0.1pt\hbox
{$\longrightarrow$}}}{#3}}

\newcommand{\infe}[1]{{#1}^{\!{\scriptscriptstyle{-\!\infty}}}}

\newcommand{\infd}[1]{{#1}^{\!\scriptscriptstyle{+\!\infty}}}

\newtheorem{theorem}{Theorem}[section]
\newtheorem{proposition}[theorem]{Proposition}

\newtheorem{lemma}[theorem]{Lemma}

\newtheorem{remark}[theorem]{Remark}
\newtheorem{example}[theorem]{Example}
\newenvironment{definition*}{\begin{trivlist}\item[\hskip
    \labelsep{\bf Definition\quad}]}%
  {\hfill\qed\end{trivlist}}
\newenvironment{notation*}{\begin{trivlist}\item[\hskip
    \labelsep{\bf Notation\quad}]}%
  {\end{trivlist}}

  \def\qed{{\unskip\nobreak\hfil\penalty50\hskip .001pt\hbox{}%
      \nobreak\hfil
      \vrule height 1.2ex width 1.1ex depth -.1ex
      \parfillskip=0pt\finalhyphendemerits=0\medbreak}}
\newenvironment{proof}{\begin{trivlist}\item[\hskip%
     \labelsep{\bf Proof.\quad}]}%
 {\hfill\qed\rm\end{trivlist}}

  {\hfill\qed\end{trivlist}}%

\begin{document}

\titlerunning{The overlap gap between left-infinite and right-infinite words}

\title{\bf The overlap gap between left-infinite\\ and right-infinite words}

\author{J. C. Costa %
  \and %
  C. Nogueira %
  \and %
  M. L. Teixeira%
}

%\date{April 27, 2018}

\maketitle

\address{ %
  J. C. Costa \& M. L. Teixeira: %
  CMAT, Dep.\ Matem\'{a}tica e Aplica\c{c}\~{o}es, Universidade do Minho, Campus
  de Gualtar, 4710-057 Braga, Portugal; %
  \email{jcosta@math.uminho.pt, mlurdes@math.uminho.pt} %
  \and %
  C. Nogueira: %
   CMAT, Escola Superior de Tecnologia e Gest\~ao,
  Instituto Polit\'ecnico de Leiria,
  Campus 2, Morro do Lena, Alto Vieiro, 2411-901
   Leiria, Portugal; %
   \email{conceicao.veloso@ipleiria.pt} %
}

\begin{abstract}
Given two finite words $u$ and $v$ of equal length, define the \emph{overlap gap between $u$ and $v$}, denoted  $og(u,v)$,  as the least integer $m$ for which there exist words $x$ and $x'$ of length $m$ such that $xu=vx'$ or $ux=x'v$. Informally, the overlap gap measures the outside parts of the greatest overlap of the given words.

For a left-infinite word $\lambda$ and a right-infinite word $\rho$, let $og_{\lambda,\rho}$ be the function defined, for each non-negative integer $n$, by $og_{\lambda,\rho}(n)=og(\lambda_n,\rho_n)$, where $\lambda_n$ and $\rho_n$ are, respectively, the suffix of $\lambda$ and the prefix of $\rho$ of length $n$. Also, denote by $OG_{\lambda,\rho}$ the image of the function $og_{\lambda,\rho}$.

In this paper, we show that $OG_{\lambda,\rho}$ is a finite set if and only if $\lambda$ and $\rho$ are ultimately periodic infinite words of the form $\lambda=\infe{u}w_1=\cdots uuuw_1$ and $\rho=w_2\infd{u}=w_2uuu\cdots$ for some finite words $u$, $w_1$ and $w_2$.

 \keywords{ Infinite word, overlap words, ultimately periodic.}
\end{abstract}

%%%%%%%%%%%%%%%%%%%%%%%%%%%%%%%%%%%%%%%%%%%%%%%%%%%%%%%%%%%%%%%%%%%%%%%%%%%%%%%%%%%%%%%%%%%%%%%%%%%%%%%%%
\section{Introduction}
 This paper deals with a combinatorial problem on infinite words that arose in the study of  semidirect products of pseudovarieties of the form $\V * \D$. Recall that a \emph{(semigroup) pseudovariety} is a class of finite semigroups closed under taking subsemigroups, homomorphic images and finite direct products, and that  \D\ is the pseudovariety of all finite semigroups whose
idempotents are right zeros.  The pseudovarieties of the form $\V * \D$ contain both the pseudovariety \D\ and its dual \K, the  pseudovariety of all finite semigroups whose idempotents are left zeros.  Free profinite semigroups relative to a given pseudovariety have been shown to be  important in revealing essential information about the pseudovariety~\cite{Almeida:1992}. Since 
 the free pro-$\D$  (resp.\  pro-$\K$) semigroups are described (see~\cite{Almeida:1992})  in terms of  finite and left-infinite  (resp.\ right-infinite) words, it is natural to expect that  combinatorial properties involving  left-infinite and  right-infinite words may arise in the study of the pseudovarieties $\V * \D$.

In~\cite{Costa&Nogueira&Teixeira:2015}, the authors  investigated a certain property, called  $\kappa$-reducibility, of semidirect products of the form $\bf V*\bf D$. In that investigation,  we wanted to choose a positive integer $n$ (determining a suffix $\lambda_n$ of length $n$ for each left-infinite word $\lambda$) in such a way that, for two  left-infinite words $\lambda$ and $\rho$, if $\lambda_n$ and $\rho_n$ would occur in a finite word, then those occurrences would be distant enough. To be more precise, given any positive integer $q$, we wanted to ensure that, if $\lambda_nx=x'\rho_n$ or $x\lambda_n=\rho_nx'$ then the length of the words $x$ and $x'$ would be necessarily at least $q$. This objective was obviously impossible in the general case. We have, however, found that if one takes a representative for each class of confinal left-infinite words (i.e.,  left-infinite words having a common suffix), then the existence of a positive integer $n$ in the above conditions is guaranteed for these representatives. This study is reviewed in Section 4 below.

In the present paper, we study a similar combinatorial question but involving a  left-infinite word and a right-infinite word instead of  two  left-infinite words. This question is  motivated by an attempt to extend the above investigation~\cite{Costa&Nogueira&Teixeira:2015} to the notion of  complete $\kappa$-reducibility (see~\cite{Almeida:2002}) of semidirect products of the form $\bf V*\bf D$.  The combinatorial question is the following:  given a  left-infinite word $\lambda$, a right-infinite word $\rho$ and a positive integer $q$, is there a positive integer $n$ such that, for the suffix $\lambda_n$  of length $n$ of $\lambda$ and the prefix $\rho_n$ of length $n$  of $\rho$,  if $\lambda_nx=x'\rho_n$ or $x\lambda_n=\rho_nx'$ then the length of the words $x$ and $x'$ are inevitably at least $q$.  The answer to this problem is negative in general, since for the infinite words of the form $\lambda=\infe{u}w_1$ and $\rho=w_2\infd{u}$  the words $x$ and $x'$ may have a length lower than the maximum of the lengths of the words $uw_1$ and $w_2u$.  In this paper,  we show that  the question above has a negative answer only for infinite words of that form. 

%%%%%%%%%%%%%%%%%%%%%%%%%%%%%%%%%%%%%%%%%%%%%%%%%%%%%%%%%%%%%%%%%%%%%%%%%%
%%%%%%%%%%%%%%%%%%%%%%%%%%%%%%%%%%%%%%%%%%%%%%%%%%%%%%%%%%%%%%%%%%%%%%%%%%
\section{Preliminaries}\label{sec:Preliminaries}
In this section, we start by briefly recalling the basic definitions and notations on finite and infinite words. We follow closely the terminology of Lothaire~\cite{Lothaire:2002}. Next, we introduce the overlap gap function associated to a pair of words (one left-infinite and the other right-infinite) that motivates this work.

%%%%%%%%%%%%%%%%%%%%%%%%%%%%%%%%%%%%%%%%%%%%%%%%%%%%%%%%%%%%%%%%%%%%%%%%%%%%%%%%%%%%%%%%%%%%%%%%%%%%%%%%%
\subsection{Finite and infinite words}\label{subsection:words}
An \emph{alphabet} is a finite non-empty set $A$ and  $A^+$ and $A^*$ denote, respectively, the \emph{free semigroup} and the \emph{free monoid} generated by  $A$, with $\varepsilon$ representing the empty word. The length of a word $w \in A^*$ is indicated by $|w|$ and the set of all words over $A$ with length $n$ is denoted by $A^n$.

 A word $u$ is called a {\em factor} (resp.\ a {\em prefix}, resp.\ a {\em suffix}) of a word $w$ if there exist words $x, y$ such that $w = xuy$ (resp.\  $w= uy$, resp.\ $w = xu$).  A word is named {\em primitive} if it cannot be written in the form $u^n$ with $n>1$.  Two words $u$ and $v$ are said to be {\em conjugate} if  $u=w_1w_2$ and $v=w_2w_1$ for some words $w_1,w_2\in A^*$.

A   {\em left-infinite} word on $A$ is a sequence $\lambda=(a_n)_{n}$ of elements of $A$ indexed by  $-\N$, also written
$\lambda=\cdots a_{-2}a_{-1}a_0$. The set of all left-infinite words on $A$ will be denoted  by $\NA$. Dually, a {\em right-infinite} word on $A$ is a sequence $\rho=(a_n)_{n}$ of letters of $A$ indexed by  $\N$, also written $\rho= a_{0}a_{1}a_2\cdots$, and $\AN$  denotes the set of all right-infinite words on $A$. A left-infinite word $\lambda$ (resp.\  a right-infinite word $\rho$) of the form $\lambda=\infe{u}w=\cdots uuuw$ (resp.\ $\rho=w\infd{u}=wuuu\cdots $), with $u\in A^+$ and $w\in A^*$, is said to be {\em ultimately periodic}.  The word $u$ (and its length $|u|$) is called  \emph{a period} of $\lambda$ (resp. of $\rho$) and the word $w$ (and its length $|w|$) is called \emph{a preperiod} of $\lambda$ (resp. of $\rho$). Notice that, for two ultimately periodic words 
$\alpha_1,\alpha_2\in\NA\cup\AN$, if $P_i=\{w\in A^+: \mbox{$w$ is a period of $\alpha_i$}\}$ ($i=1,2$), then either $P_1=P_2$ or $P_1\cap P_2=\emptyset$.

For any ultimately periodic word $\lambda\in\NA$ (resp.\ $\rho\in\AN$) there exist unique words $u,w$ of shortest length such that $\lambda=\infe{u}w$ (resp.\ $\rho=w\infd{u}$). Of course, in this case, $u$ is a primitive word and, if $w\neq \varepsilon$, the first (resp.\ the last)  letters of $u$ and $w$ do not coincide. Besides, $\lambda=\infe{u}w$ (resp.\ $\rho=w\infd{u}$) is said to be the \emph{canonical form} of $\lambda$  (resp. of $\rho$), the word    $u$ (and its length) is called  \emph{the period} of $\lambda$  (resp. of $\rho$) and  the word $w$ (and its length) is called \emph{the preperiod} of $\lambda$  (resp. of $\rho$). Obviously, ultimately periodic words  may have the same period number and distinct period words. If the preperiod number of an ultimately periodic word $\alpha\in\NA\cup\AN$ is $0$ and $p\in \N$ is a period of $\alpha$, then $\alpha$ is said to be a \emph{$p$-periodic} word.

An integer $p \geq 1$ is a \emph{period} of a finite word $w = a_1a_2\cdots a_n$ where $a_i\in A$ if $a_i=a_{i+p}$ for $i=1,\ldots,n-p$. In such a case, $w$ is called a \emph{$p$-periodic} word. Notice that a finite word $w\in A^+$ is $p$-periodic if and only if $w$ is a factor of a $p$-periodic infinite word.

%%%%%%%%%%%%%%%%%%%%%%%%%%%%%%%%%%%%%%%%%%%%%%%%%%%%%%%%%%%%%%%%%%%%%%%%%%%%%%%%%%%%%%%%%%%%%%%%%%%%%%%%%
\subsection{The overlap gap function}\label{subsection:overlap-gap-function}
For words   $u$ and $v$ of equal length, we define:
$$\begin{array}{rcl}
log(u,v)&\hspace*{-2mm}=\hspace*{-2mm}&\mbox{min}\{n\in\N:  ux=x'v\mbox{ with }x,x'\in A^n\},\\
rog(u,v)&\hspace*{-2mm}=\hspace*{-2mm}&\mbox{min}\{n\in\N: xu=vx'\mbox{ with }x,x'\in A^n\},\\
og(u,v)&\hspace*{-2mm}=\hspace*{-2mm}&\mbox{min}\{log(u,v),rog(u,v)\}.
\end{array}$$
 The non-negative integers $log(u,v)$,  $rog(u,v)$ and  $og(u,v)$ are called, respectively, the  \emph{left overlap gap},   the \emph{right overlap gap} and the \emph{overlap gap} between $u$ and $v$. Notice that $log(u,v)=0$ if and only if $u=v$ if and only if $rog(u,v)=0$.

Consider infinite words $\lambda\in \NA$   and  $\rho\in\AN$. For any non-negative integer $m$, we will denote by $\lambda_m$ and $\rho_m$, respectively, the suffix of $\lambda$ and the prefix of $\rho$ of length $m$. For each $f\in\{log,rog,og\}$, we let $f_{\lambda,\rho}$ be the function $f_{\lambda,\rho}: \N\rightarrow \N$ defined, for each non-negative integer $n$, by $f_{\lambda,\rho}(n)=f(\lambda_n,\rho_n)$.  The images $log_{\lambda,\rho}(\N)$, $rog_{\lambda,\rho}(\N)$  and $og_{\lambda,\rho}(\N)$ of the functions $log_{\lambda,\rho}$, $rog_{\lambda,\rho}$  and $og_{\lambda,\rho}$, will be denoted, respectively, by  $LOG_{\lambda,\rho}$, $ROG_{\lambda,\rho}$ and $OG_{\lambda,\rho}$. Usually, when the words $\lambda$ and $\rho$ are clear from the context they will be omitted in the above notations. 

\begin{example}\label{example:periodic}
Consider the left-infinite word $\lambda=\infe{(baa)}$   and  the right-infinite word $\rho=\infd{(aab)}$ over the alphabet $A=\{a,b\}$. Then, for every $k\in\N$,
$$\begin{array}{l}
log(0)=log(1)=\ log(2+3k)=0,\ log(3+3k)=1,\ log(4+3k)=2,\\
rog(0)=rog(1)=\ rog(2+3k)=0,\ rog(3+3k)=2,\ rog(4+3k)=1,\\
og(0)=og(1)=\ og(2+3k)=0,\ og(3+3k)=\ og(4+3k)=1.
\end{array}$$
Therefore, $LOG=ROG=\{0,1,2\}$ and $OG=\{0,1\}$.
\end{example}

The following observation is useful.

\begin{lemma}\label{lemma:logn+1}
Let $\lambda\in \NA$   and  $\rho\in\AN$. For every $n\in\N$,
$$log(n+1)\in\{0,log(n)+1\},\quad rog(n+1)\geq rog(n)-1.$$
\end{lemma}
\begin{proof} Suppose $\lambda=(a_n)_{n\in-\N}$   and  $\rho=(b_n)_{n\in\N}$. For a fixed $n\in\N$ let $p=log(n)$ and $q=rog(n)$. Then $\lambda_nx=x'\rho_n$ and  $y\lambda_n=\rho_ny'$ for some $x,x'\in A^p$ and $y,y'\in A^q$. Notice that $\lambda_{n+1}=a_n\lambda_n$ and $\rho_{n+1}=\rho_nb_n$.

From $\lambda_nx=x'\rho_n$  it follows that $a_n\lambda_nxb_n=a_nx'\rho_nb_n$, that is, $\lambda_{n+1}xb_n=a_nx'\rho_{n+1}$. Whence, $log(n+1)\leq p+1$. Suppose that $log(n+1)\not\in\{0, p+1\}$. Hence
$\lambda_{n+1}z=z'\rho_{n+1}$ for some words $z$ and $z'$ with $1\leq |z|=|z'|\leq p$. Then $z=wb_n$ and $z'=a_nw'$ for some words $w$ and $w'$, whence $\lambda_{n}w=w'\rho_{n}$. From this equality one deduces that $log(n)<p$ which contradicts the assumption that  $p=log(n)$. Therefore  $log(n+1)\in\{0, p+1\}=\{0,log(n)+1\}$.

Now, the condition $ rog(n+1)\geq rog(n)-1$ is obviously true for $q= rog(n)\leq 1$. So, we assume   $ q\geq 2$. Suppose that $rog(n+1)< q-1$. Then $z\lambda_{n+1}=\rho_{n+1}z'$ for some words $z$ and $z'$ with $|z|=|z'|< q-1$. Then $w\lambda_{n}=\rho_{n}w'$, where $w=za_n$ and $w'=b_nz'$, meaning that $rog(n)<q$. This contradicts the assumption  $q=rog(n)$ and, so, we conclude that $ rog(n+1)\geq rog(n)-1$.
\end{proof}

\begin{example}\label{example:ult-periodic}
Consider the left-infinite word $\lambda=\infe{(b^2a^2)}cab$   and  the right-infinite word $\rho=\infd{(ab^2a)}$ over the alphabet $A=\{a,b,c\}$. The values of the functions $log$, $rog$ and $og$ for these words $\lambda$ and $\rho$ are shown in the following table 

\begin{footnotesize}
$$\begin{array}{|c||c|c|c|c|c|c|c|c|c|c|c|c|c|c|c|c|}\hline
\rule[-1mm]{0mm}{5mm}{n} & \;\! 0\;\!  &\;\!  1\;\!  & \;\! 2 \;\!  &\;\!  3\;\!  &\;\!  4\;\!  & \;\! 5 \;\! & \;\! 6\;\!  & \;\! 7 \;\! & \;\! 8 \;\! &\;\! 9\;\! & 10 & 11 & 12 & 13 & 14 &  \cdots \\
\hline\hline

\rule[-1mm]{0mm}{5mm}{log(n)}  & 0 & 1 & 0  & 1 & 2 & 3 & 4 & 5 & 6 & 7 & 8 & 9 & 10 & 11 & 12 & \cdots\\
\hline
\rule[-1mm]{0mm}{5mm}{rog(n)}  & 0 & 1 & 0  & 3 & 3 & 3 & 5 & 5 & 4 & 3 & 6 & 5 & 4 & 3 & 6 &\cdots\\
\hline
\rule[-1mm]{0mm}{5mm}{og(n)}  & 0 & 1 & 0  & 1 & 2 & 3 & 4 & 5 & 4 & 3 & 6 & 5 & 4 & 3 & 6 &\cdots\\
\hline
\end{array}$$
\end{footnotesize}
\smallskip

So, $LOG=\N$, $ROG=\{0,1,3,4,5,6\}$ and $OG=\{0,1,2,3,4,5,6\}$.
\end{example}

%%%%%%%%%%%%%%%%%%%%%%%%%%%%%%%%%%%%%%%%%%%%%%%%%%%%%%%%%%%%%%%%%%%%%%%%%%
%%%%%%%%%%%%%%%%%%%%%%%%%%%%%%%%%%%%%%%%%%%%%%%%%%%%%%%%%%%%%%%%%%%%%%%%%%
\section{Characterization of overlap functions with finite image}\label{sec:characterization}
The purpose of this paper is to characterize the pairs of words in $\NA\times\AN$ for which the overlap gap function has finite image.

%%%%%%%%%%%%%%%%%%%%%%%%%%%%%%%%%%%%%%%%%%%%%%%%%%%%%%%%%%%%%%%%%%%%%%%%%%
%%%%%%%%%%%%%%%%%%%%%%%%%%%%%%%%%%%%%%%%%%%%%%%%%%%%%%%%%%%%%%%%%%%%%%%%%%
\subsection{Overlap gap versus right overlap gap}\label{subsec:og_versus_rog}
We begin by showing that $og_{\lambda,\rho}(\N)$ being finite or not depends  exclusively on the function $rog_{\lambda,\rho}$. 

\begin{proposition}\label{prop:og_vs_rog}
Let $\lambda\in \NA$   and  $\rho\in\AN$. The set $OG_{\lambda,\rho}$ is finite  if and only if the set $ROG_{\lambda,\rho}$ is finite.
\end{proposition}
\begin{proof}
 If $ROG$ is a finite set, then it is obvious that $OG$ is also finite since, by definition,   for every $n\in\N$, $og(n)=\mbox{min}\{log(n),rog(n)\}$.  In order to prove the inverse implication, assume that $OG$ is a finite set and suppose that $ROG$ is infinite.  Let $M$ be the maximum element of $OG$ and let $n\in\N$ be such that $rog(n)>2M$. Hence, Lemma~\ref{lemma:logn+1} implies that, for every $i\in\{0,1,\ldots,M\}$,
\begin{equation}\label{eq:rogn+1}
rog(n+i)\geq rog(n)-i>M.
\end{equation}
 So, by definition of $og$  and $M$, it follows that $log(n+i)\leq M$ for all  $i\in\{0,1,\ldots,M\}$. If all such $log(n+i)$ would be not null one would deduce from  Lemma~\ref{lemma:logn+1} that $log(n+M)=log(n)+M>M$. As this does not happen by hypothesis, one deduces that $log(n+j)=0$ for some  $j\in\{0,1,\ldots,M\}$. As, for every $k\in\N$, $log(k)=0$ if and only if $rog(k)=0$, it follows that  $rog(n+j)=0$. This is in contradiction with~\eqref{eq:rogn+1} and, so, we conclude that  $ROG$ cannot be infinite, thus concluding the proof of the proposition.
\end{proof}
\begin{remark}\label{remark:upper_bound_ROG}
It should be noted that it follows from the above proof that, if  $M$ is an upper bound for $OG_{\lambda,\rho}$, then $2M$ is an upper bound for $ROG_{\lambda,\rho}$.
\end{remark}

As a result of Proposition~\ref{prop:og_vs_rog},  we will now focus our attention on the function  $rog_{\lambda,\rho}$.

%%%%%%%%%%%%%%%%%%%%%%%%%%%%%%%%%%%%%%%%%%%%%%%%%%%%%%%%%%%%%%%%%%%%%%%%%%
%%%%%%%%%%%%%%%%%%%%%%%%%%%%%%%%%%%%%%%%%%%%%%%%%%%%%%%%%%%%%%%%%%%%%%%%%%
\subsection{The rog function for ultimately periodic words}\label{subsec:rog_for_upw}

In Examples~\ref{example:periodic} and~\ref{example:ult-periodic}, we have considered ultimately periodic words $\lambda\in \NA$   and  $\rho\in\AN$  with the same set of period words. In both examples $\big(rog_{\lambda,\rho}(n)\big)_{n\in\N}$ is an ultimately periodic sequence.  This is  true in general for words of that form, as we shall prove next.

In this section, we fix  two ultimately periodic words $\lambda\in \NA$  and  $\rho\in\AN$ with equal sets of period words
 and let 
$$\lambda=\infe{u}w_1\quad\mbox{and}\quad\rho=w_2\infd{v}$$
 be their canonical forms. It is clear that  the period words $u$ and $v$ of $\lambda$  and  $\rho$ are conjugate words, say $u=yz$ and $v=zy$ with $z\in A^*$ and $y\in A^+$, so that the period numbers $|u|$ and $|v|$ of, respectively, $\lambda$   and  $\rho$ are the same integer $p\geq 1$.  Let $$m'=\mbox{min}\{|w_1|,|w_2|\},\quad  m=\mbox{max}\{|w_1|,|w_2|\}\quad\mbox{and}\quad M=m+p-1.$$ 
 The other case being symmetric, we consider only the case in which $m=|w_1|$ and  $m'=|w_2|$. 

We begin by showing that  $M$ is an upper bound for $rog_{\lambda,\rho}$.
\begin{lemma}\label{lemma:rog-ult-period}
For every $n\in\N$, $rog_{\lambda,\rho}(n)\leq M$.
\end{lemma}
\begin{proof}  That the condition $rog(n)\leq M$ holds for every $n\leq M$ is trivial. Consider now an $n>M$ so that  $\lambda_n$ and $\rho_n$ can be written in the forms $$\lambda_n=u_1u^{\ell_1}uw_1\quad\mbox{and}\quad \rho_n=w_2zu^{\ell_2}u_2,$$
where $\ell_1$ and $\ell_2$ are non-negative integers and $u_1$ and $u_2$ are, respectively, a proper suffix and a proper prefix of $u$. 

If  $u_2=\varepsilon$, then $|zu^{\ell_2}|\geq |u_1u^{\ell_1}u|$. As both words $zu^{\ell_2}$ and $u_1u^{\ell_1}u$ are suffixes of a power of $u$, it follows that  $u_1u^{\ell_1}u$  is a suffix of $zu^{\ell_2}$, whence $zu^{\ell_2}=u'_1u_1u^{\ell_1}u$ for some word $u'_1$. Therefore,  $w_2u'_1\lambda_n=w_2u'_1u_1u^{\ell_1}uw_1=w_2zu^{\ell_2}w_1=\rho_nw_1$ and so $rog(n)\leq |w_1|=m\leq M$. 

Suppose now that  $u_2\neq\varepsilon$ and let $u'_2$ be the proper suffix of $u$ such that $u=u_2u'_2$. If $|u_1u^{\ell_1}|\leq |u^{\ell_2}|$, then $u^{\ell_2}=u'_1u_1u^{\ell_1}$ for some word $u'_1$. Then $w_2zu'_1\lambda_n=w_2zu'_1u_1u^{\ell_1}uw_1=w_2zu^{\ell_2}u_2u'_2w_1=\rho_nu'_2w_1$, so that $rog(n)\leq |u'_2w_1|\leq p-1+m= M$.  Assume next that $|u_1u^{\ell_1}|> |u^{\ell_2}|$, whence $u_1u^{\ell_1}=u'_1u^{\ell_2}$ for some word $u'_1$. We notice that in this case $|u'_1|< |z|$ since $n=|u_1u^{\ell_1}uw_1|=|w_2zu^{\ell_2}u_2|$ and $|uw_1|> |w_2u_2|$. Therefore, $|u'_1|< |u|$ and thus $u_1=u'_1$ and $\ell_1=\ell_2$. As $u_1$ and $z$ are suffixes of $u$, we deduce that $z=z'u_1$ for some word $z'$ and, as consequence, that $zu^{\ell_2}=z'u_1u^{\ell_1}$. It follows that $w_2z'\lambda_n=w_2z'u_1u^{\ell_1}uw_1=w_2zu^{\ell_2}u_2u'_2w_1=\rho_nu'_2w_1$. This shows that  $rog(n)\leq M$ also in the present case, thus concluding the proof of the lemma.
 \end{proof} 

 Proposition~\ref{prop:rog-ult-period} below proves the  ultimately periodic nature of the function $rog_{\lambda,\rho}$. It shows in particular that for every  integer $n$ greater than a certain threshold $n_0$, $rog_{\lambda,\rho}(n)\in\{m,m+1,\ldots,M\}$.  We are not interested in determining the optimal threshold $n_0$ for each specific pair $(\lambda,\rho)$. However, we have tried to minimize grossly a threshold that works for a generic pair $(\lambda,\rho)$. Such a threshold is, of course, at least $m$. Although Examples~\ref{example:periodic} and~\ref{example:ult-periodic} may suggest that the  threshold could be close to $m+p$, this is not true in general. For instance, for an arbitrary integer $i\geq 1$, if $\lambda=\infe{a}ba^i$ and $\rho=a^ib\infd{a}$, then $rog(2i+1)=0<i+1=m$. 

 Let $\overline m$ be the (unique) integer in the interval $\{M,M+1,\ldots,M+p-1\}$ such that $\overline m\equiv  m'+|z|\;(\mbox{mod }p)$ and notice that, as $m'+|z|\leq M$, $\overline m=  m'+|z|+dp$ for some non-negative integer $d$. Let $\overline w_2$ be the prefix of $\rho$ of length $\overline m$, so that $\overline w_2=w_2zu^d$ and $\rho=\overline w_2\infd{u}$. Finally, let $k_0$ be the least positive integer such that $|z|+(d+k_0)p>M$.
\begin{proposition}\label{prop:rog-ult-period} 
For all integers $k\geq k_0$ and  $j\in\{0,1,\ldots,p-1\}$, $$rog_{\lambda,\rho}(\overline m+1+j+pk)=M-j.$$
\end{proposition}
\begin{proof} 
Consider integers $k\geq k_0$ and  $j\in\{0,1,\ldots,p-1\}$ and let $n=\overline m+1+j+pk$.    Notice that $M< \overline m+1\leq M+p$, whence  $\lambda_n$ and $\rho_n$ can be written in the forms 
$$\lambda_n=u_1u^{k+1}w_1\quad\mbox{and}\quad \rho_n=\overline w_2u^{k}u_2,$$
where  $u_1$ is a  suffix of $u^2$ and $u_2$ is a prefix of $u$ with $|u_1|=\overline m-M+j$ and $|u_2|=1+j$. We now observe that $|\overline w_2|-|u_1|=M-j\geq m\geq m'=|w_2|$. Hence $|u_1|\leq |zu^d|$ and so, as both $u_1$ and $zu^d$ are suffixes of a power of $u$, $u_1$ is a suffix of $zu^d$.  Let $x=w_2u'_1$ and $x'=u'_2w_1$, where $u'_1$ and $u'_2$ are the words defined by $zu^d=u'_1u_1$ and $u=u_2u'_2$. Then
\begin{equation}\label{eq:xlambda=rhox'}
x\lambda_n=w_2u'_1u_1u^{k+1}w_1=w_2zu^{d+k+1}w_1=\overline w_2u^{k}u_2u'_2w_1=\rho_nx'.
\end{equation}
As $|x|=|x'|= M-j$, this shows that $rog(n)\leq M-j$. 

Let  $q=rog(n)$. Then $x_1\lambda_{n}=\rho_{n}x'_1$, that is, $x_1u_1u^{k+1}w_1= w_2zu^{d+k}u_2x'_1$,   for some $x_1,x'_1\in A^q$. By definition of $k_0$ and since $k\geq k_0$, $|zu^{d+k}|=|z|+(d+k)p\geq |z|+(d+k_0)p>M$.  It follows that $|zu^{d+k}u_2x'_1|\geq |uw_1|$ and, so, $uw_1$ is a suffix of $zu^{d+k}u_2x'_1$. Suppose that $q<m$, whence $m>0$ and $w_1\neq\varepsilon$. If  $a$ is the first letter of $w_1$, then $ua$ is a factor of the $p$-periodic word $zu^{d+k}u_2$. This is in contradiction with the fact that $a$ is not the first letter of $u$ because $\lambda=\infe{u}w_1$ is in canonical form. Therefore  $q\geq m$. This  already proves the result for $j=p-1$. So, we assume $j\neq p-1$ and suppose that  $m\leq q<M-j$. As $x'_1$ is a suffix of $\lambda_{n}$ of length $q\geq m$, it follows that $x'_1=x'_2w_1$  and $x_1u_1u^{k+1}= w_2zu^{d+k}u_2x'_2$ for some word $x'_2$. Since we are assuming $|x'_1|<|x'|$,  we have further $|x'_2|<|u'_2|$. Now, as $k\geq k_0\geq 1$, we deduce that $uu_2x'_2$  is a suffix of $u^2$. Denoting  $u''=u_2x'_2$, it follows that $u=u'u''=u''u'$ for some word $u'$. Since $u'$ and $u''$ are both non-empty, this is a contradiction with the fact that $u$ is a primitive word  because $\lambda=\infe{u}w_1$ is in canonical form. This proves that $q$ must be equal to $M-j$ and concludes the proof of the proposition.
\end{proof} 
%%%%%%%%%%%%%%%%%%%%%%%%%%%%%%%%%%%%%%%%%%%%%%%%%%%%%%%%%%%%%%%%%%%%%%%%%%
%%%%%%%%%%%%%%%%%%%%%%%%%%%%%%%%%%%%%%%%%%%%%%%%%%%%%%%%%%%%%%%%%%%%%%%%%%
\subsection{The main result}\label{subsec:main-result}

In the last section, we have shown that ultimately periodic words  with the same set of period words determine right overlap gap functions with finite image. We now  prove that this happens only for that kind of infinite words.

\begin{theorem}\label{theo:finite_rog}
Consider two infinite  words $\lambda\in \NA$   and  $\rho\in\AN$. The set $ROG_{\lambda,\rho}$ is  finite if and only if $\lambda=\infe{u}w_1$ and $\rho=w_2\infd{u}$ for some words  $u,w_1,w_2\in A^*$.
\end{theorem}
\begin{proof} Let $\lambda=(a_n)_{n\in-\N}$   and  $\rho=(b_n)_{n\in\N}$.  If $\lambda$ and $\rho$ are of the forms  $\lambda=\infe{u}w_1$ and $\rho=w_2\infd{u}$, then $ROG$ is finite by   Lemma~\ref{lemma:rog-ult-period}. It remains to prove the direct implication of the theorem. So, let us assume that $ROG$ is a finite set. In general, we denote by $\overline{ROG}_{\lambda,\rho}$ the subset of $ROG_{\lambda,\rho}$ formed by the elements that are image under $rog_{\lambda,\rho}$ of an infinite number of non-negative integers. Let $m=\mbox{min}\;\overline{ROG}$ and 
$M=\mbox{max}\;\overline{ROG}$. We reduce to the case in which $m=0$. Suppose that $m\neq 0$ and let $\lambda'\in \NA$   and  $\rho'\in\AN$ be the words such that
$$\lambda=\lambda'\lambda_m\quad\mbox{and}\quad \rho=\rho_m\rho',$$
that is, $\lambda'=\cdots a_{m+1}a_m$ and $\rho'=b_mb_{m+1}\cdots$. Denote by $rog'$ the function $rog_{\lambda',\rho'}$. We claim that
$$\forall n\in\N\ \Big(rog(n)\geq m\ \Rightarrow\ rog'(n-m)=rog(n)-m\Big).$$
In fact, for an $n\in \N$, if $rog(n)=p\geq m$, then $n\geq m$ and $xy\lambda_n=\rho_ny'x'$ for some words $x,x'\in A^m$ and $y,y'\in A^{p-m}$. Therefore $y\lambda'_{n-m}=\rho'_{n-m}y'$, whence $\ rog'(n-m)\leq p-m$. On the other hand, if  $\ rog'(n-m)=q< p-m$, then  $z\lambda'_{n-m}=\rho'_{n-m}z'$ for some $z,z'\in A^q$ and, so, $\rho_mz\lambda'_{n-m}\lambda_m=\rho_m\rho'_{n-m}z'\lambda_m$. Thus  $w\lambda_{n}=\rho_{n}w'$, with $w=\rho_mz$ and $w'=z'\lambda_m$, whence $rog(n)\leq m+q<p$, in contradiction with the assumption  that $rog(n)=p$. Therefore $rog'(n-m)=p-m=rog(n)-m$, which proves the claim. In particular, it follows that $\overline{ROG}_{\lambda',\rho'}=\{j-m\mid j\in \overline{ROG}_{\lambda,\rho}\}$, whence $m'=\mbox{min}\;\overline{ROG}_{\lambda',\rho'}=0$ and 
$M'=\mbox{max}\;\overline{ROG}_{\lambda',\rho'}=M-m$. It is clear that, if the direct implication of the theorem holds for the pair $(\lambda',\rho')$, then it also holds for the pair  $(\lambda,\rho)$. We may therefore assume that $m=0$.

 Let $(n_i)_{i\in\N}$ be the crescent succession of all natural numbers $n_i$ such that  $rog(n_i)=0$ (meaning that $\lambda_{n_{i}}$ and $\rho_{n_{i}}$ are the same word, which we denote by $t_i$). For each $i\in\N$, we have $t_{i+1}=u_it_{i}=t_{i}v_i$ for some words $u_i,v_i\in A^{k_i}$ where $k_i=n_{i+1}-n_i$. From the equality $u_it_{i}=t_{i}v_i$ one deduces that $t_i$ is a $k_i$-periodic word with
$$t_i=(r_is_i)^{\ell_i}r_i,\quad u_i=r_is_i,\quad v_i=s_ir_i,$$
for some $r_i\in A^*$, $s_i\in A^+$ and $\ell_i\in\N$. 

Suppose there is a $k\in \N$ and an infinite subset $I$ of $\N$ such that $k_i=k$ for all $i\in I$. Then all $u_i$ ($i\in I$) are the same word $u$ and all $v_i$ are the same word $v$, with $u=rs$ and $v=sr$ for some words $r$ and $s$. Moreover $\lambda=\infe{v}=\infe{u}r$ and $\rho=\infd{u}$. Therefore, the result holds in this case. 

Otherwise, there is  an infinite subset $I$ of $\N$ such that, for every $i\in I$, $k_i>M$ and $n_i>2M$. Let $q= M+1$ and let $i\in I$. We claim that in this case $rog(n_{i+1}-q)=q$. Since $I$ is infinite, this claim entails  that $M<q\in \overline{ROG}$, a statement inconsistent with the definition of $M$. So, actually, we will be able to deduce  that the present case is impossible, thus completing the proof of the direct implication.

 Let us show the claim. Recall that $\lambda_{n_{i+1}}=u_i\lambda_{n_{i}}=\rho_{n_{i}}v_i=\rho_{n_{i+1}}$ and $\lambda_{n_{i}}=\rho_{n_{i}}=(r_is_i)^{\ell_i}r_i$, with $u_i=r_is_i$ and $v_i=s_ir_i$ in $A^{k_i}$. Since $k_i>M$ (whence $k_i\geq q$), we have $u_i=u'u''$ and $v_i=v''v'$ for some words $u',v'\in A^q$ and $u'',v''\in A^{k_i-q}$. Hence $$\lambda_{n_{i+1}-q}=u''\lambda_{n_{i}}=(u''u')^{\ell_i}u''r_i\quad\mbox{and}\quad\rho_{n_{i+1}-q}=\rho_{n_{i}}v''=r_iv''(v'v'')^{\ell_i} $$ 
are $k_i$-periodic words.
 Once $u_i$ is a prefix of $\rho_{n_{i+1}}$, we have $u_i=\rho_{k_i}=b_0\cdots b_{k_i-1}$, whence $u'=b_0\cdots b_{q-1}$ and $u''=b_q\cdots   b_{k_i-1}$. Since $\rho_{n_{i+1}-q}$ is a prefix of $\rho$, it follows that
$$\lambda_{n_{i+1}-q}=(b_q\cdots   b_{k_i-1}b_0\cdots  b_{q-1})^{\ell_i}u''r_i\quad\mbox{and}\quad\rho_{n_{i+1}-q}=(b_0\cdots   b_{k_i-1})^{\ell_i}r_iv''.$$
Moreover $u'\lambda_{n_{i+1}-q}=u_i\lambda_{n_{i}}=\rho_{n_{i}}v_i=\rho_{n_{i+1}-q}v'$, whence  $rog(n_{i+1}-q)\leq q$. Suppose that  $rog(n_{i+1}-q)=p<q$, so that $x\lambda_{n_{i+1}-q}=\rho_{n_{i+1}-q}x'$ for some $x,x'\in A^p$. Notice that $p+k_i\leq  n_{i+1}-q$ since $n_i\geq 2M+1\geq p+q$. It follows that $xu''u'$ is a prefix of the $k_i$-periodic word $\rho_{n_{i+1}-q}$. As a consequence, we have $b_q\cdots   b_{k_i-1}b_0\cdots  b_{q-1}=u''u'=b_p\cdots   b_{k_i-1}b_0\cdots  b_{p-1}$ and, so, $u_i=b_0\cdots   b_{k_i-1}= b_{k_i-q+p}\cdots   b_{k_i-1}b_0\cdots b_{k_i-q+p-1}$.

 Let $z_1=b_0\cdots b_{k_i-q+p-1}$ and $z_2=b_{k_i-q+p}\cdots   b_{k_i-1}$. It is well known that the equality $z_1z_2=z_2z_1$ implies the existence of a word $z$ and positive integers $h_1$ and $h_2$ such that $z_1=z^{h_1}$ and $z_2=z^{h_2}$. It follows that $u_i=z^{h_1+h_2}$. Then $\lambda_{n_{i+1}}=\rho_{n_{i+1}}=z^{h}z'$ for some positive integer $h$ and some proper prefix $z'$ of $z$. Thus $\lambda_{n_{i+1}-|z|}=\rho_{n_{i+1}-|z|}=z^{h-1}z'$, whence  $rog(n_{i+1}-|z|)=0$.  This is absurd  by the definition of the sequence $(n_i)_{i\in\N}$, because  $n_i<n_{i+1}-|z|<n_{i+1}$ since $0\neq |z|<k_i$. Therefore  $rog(n_{i+1}-q)=q$, thus proving the claim. As argued above, this shows that the present case is impossible and concludes the proof of the  theorem.
\end{proof}

Combining Proposition~\ref{prop:og_vs_rog} with Theorem~\ref{theo:finite_rog}, we deduce the main result of the paper.
\begin{theorem}\label{theo:main}
For two given  words $\lambda\in \NA$   and  $\rho\in\AN$, the set $OG_{\lambda,\rho}$ is  finite if and only if $\lambda=\infe{u}w_1$ and $\rho=w_2\infd{u}$ for some words  $u,w_1,w_2\in A^*$.
\end{theorem}

%%%%%%%%%%%%%%%%%%%%%%%%%%%%%%%%%%%%%%%%%%%%%%%%%%%%%%%%%%%%%%%%%%%%%%%%%%
%%%%%%%%%%%%%%%%%%%%%%%%%%%%%%%%%%%%%%%%%%%%%%%%%%%%%%%%%%%%%%%%%%%%%%%%%%
\section{The one-sided version}\label{sec:one-sided}

 Theorem~\ref{theo:main} above  characterizes the pairs of words in $\NA\times\AN$ for which the overlap gap function has finite image.  We now state and prove a  one-sided version of this result,  involving  two left-infinite words $\lambda$ and $\rho$. It should be noted that, in this new situation, the definitions in Section~\ref{subsection:overlap-gap-function} need to be adapted. To do this, simply change the definition of  $\rho_n$, which now denotes the suffix of length $n$ of $\rho$. This is a much simpler problem that was already treated in~\cite{Costa&Nogueira&Teixeira:2015} but with a slightly different definition of the overlap gap function and not announced with full generality. 

\begin{theorem}\label{theo:one-sided}
Consider two left-infinite  words $\lambda,\rho\in\NA$. The set $OG_{\lambda,\rho}$ is  finite if and only if $\rho=\lambda w$ or $\lambda=\rho w$ for some word $w\in A^*$.
\end{theorem}
\begin{proof}
Suppose first that  $\rho=\lambda w$   or $\lambda=\rho w$ for some word $w\in A^*$. Let $k=|w|$ and let $n$ be an arbitrary non-negative integer. In case  $\rho=\lambda w$, one has $$\lambda_nw=\rho_{n+k}=w'\rho_n$$ where $w'$ is the prefix of $\rho_{n+k}$ of length $k$. This shows that $log_{\lambda,\rho}(n)\leq k$ and so, since $n$ is arbitrary, that  $LOG_{\lambda,\rho}$ is  a finite set. In case $\lambda=\rho w$, one deduces symmetrically that $rog_{\lambda,\rho}(n)\leq k$ and that $ROG_{\lambda,\rho}$ is  finite. It follows, in both cases, that $OG_{\lambda,\rho}$ is  a finite set.

 Suppose now that $OG_{\lambda,\rho}$ is a finite set. Then, there exists a constant $k\in \N$ such that  $og_{\lambda,\rho}(n)=k$ for an infinite number of non-negative integers $n$.  By definition,   $og_{\lambda,\rho}(n)=\mbox{min}\{log_{\lambda,\rho}(n),rog_{\lambda,\rho}(n)\}$  for every $n\in\N$. Hence,  one  of the functions $log_{\lambda,\rho}$ and  $rog_{\lambda,\rho}$ take the same value $k$ for an infinite sequence $n_0<n_1<\cdots<n_i<\cdots$ of  natural numbers $n_i$. In the first case, that means that, for each $i\in\N$, $$\lambda_{n_i}x_i=x'_i\rho_{n_i}$$ for some $x_i,x'_i\in A^k$. Therefore, $x_i=\rho_k$ for every $i$ such that $n_i\geq k$. Moreover, denoting $\rho_k$ by $w$,  $\lambda w=\rho$.  In the second case one deduces symmetrically that $\lambda=\rho w$, with $w=\lambda_k$. This concludes the proof of the theorem.
\end{proof}

A symmetric result is valid for right-infinite words.

%%%%%%%%%%%%%%%%%%%%%%%%%%%%%%%%%%%%%%%%%%%%%%%%%%%%%%%%%%%%%%%%%%%%%%%%%%%%%%%%%%%%%%%%%%%%%
%%%%%%%%%%%%%%%%%%%%%%%%%%%%%%%%%%%%%%%%%%%%%%%%%%%%%%%%%%%%%%%%%%%%%%%%%%%%%%%%%%%%%%%%%%%%
\section*{Acknowledgments}
This work  was supported by the European Regional Development Fund, through the programme COMPETE, and by the
Portuguese Government through FCT -- \emph{Funda\c c\~ao para a Ci\^encia e a Tecnologia}, under the project
 PEst-C/MAT/UI0013/2014.

 %%%%%%%%%%%%%%%%%%%%%%%%%%%%%%%%%%%%%%%%%%%%%%%%%%%%%%%%%%%%%%%%%%%%%%%%%%%%%%%%%%%%%%%%%%%%%
%%%%%%%%%%%%%%%%%%%%%%%%%%%%%%%%%%%%%%%%%%%%%%%%%%%%%%%%%%%%%%%%%%%%%%%%%%%%%%%%%%%%%%%%%%%%

\end{document}